\documentclass{amsart}
\usepackage{verbatim}
\usepackage{rotating} 
\usepackage{amssymb}
\usepackage{mathrsfs,amsfonts,amsmath,amssymb,epsfig,amscd,xy,amsthm,pb-diagram} 

\newtheorem{theorem}{Theorem}[section]
\newtheorem{proposition}[theorem]{Proposition}

\newtheorem{lemma}[theorem]{Lemma}

\newtheorem{corollary}[theorem]{Corollary}
\newtheorem{question}[theorem]{Question}

\theoremstyle{plain}

\theoremstyle{remark}

\newcommand{\F}{{\mathbb F}}
\newcommand{\Q}{{\mathbb Q}}

\newcommand{\Z}{{\mathbb Z}}
\newcommand{\N}{{\mathbb N}}

\newcommand{\cA}{{\mathcal A}}

\newcommand{\cK}{{\mathcal K}}

\DeclareMathOperator{\rad}{rad}

\newcommand{\bF}{{\mathbb F}}

\newcommand{\cP}{\mathcal{P}}

\newcommand{\scrX}{\mathscr{X}}

\newcommand{\scrS}{\mathscr{S}}

\author{Jason P.~Bell}
\address{
Jason P.~Bell\\
University of Waterloo\\
Department of Pure Mathematics\\
Waterloo, Ontario, Canada N2L 3G1}
\email{jpbell@uwaterloo.ca}

\author{Khoa D.~Nguyen}
\address{
Khoa D.~Nguyen \\
Department of Mathematics and Statistics\\
University of Calgary\\
AB T2N 1N4, Canada
}
\email{dangkhoa.nguyen@ucalgary.ca}


\begin{document}
	\title[Ruzsa's conjecture for polynomials]{An analogue of Ruzsa's conjecture for polynomials over finite fields}
	\date{}
	\begin{abstract}
	In 1971, Ruzsa conjectured that if $f:\ \N\rightarrow\Z$ with
	$f(n+k)\equiv f(n)$ mod $k$ for every $n,k\in\N$ and $f(n)=O(\theta^n)$ with $\theta<e$ then $f$ is a polynomial. In this paper, we investigate the analogous problem for the ring of polynomials over a finite field.		 
	\end{abstract}
	
	\maketitle
	
	\section{Introduction}
		Let $\N$ denote the set of positive integers and let $\N_0=\N\cup\{0\}$. 
		A strong form of a conjecture by Ruzsa is the following assertion.
		Suppose that  $f:\ \N_0\rightarrow \Z$ satisfies 
		the following 2 properties:
		\begin{itemize}
		\item [(P1)] $f(n+p)\equiv f(n)$ mod $p$ for 
		every prime $p$ and every $n\in\N_0$;
		\item [(P2)] $\displaystyle\limsup_{n\to\infty}\frac{\log\vert f(n)\vert}{n}<e$.
		\end{itemize}
		Then $f$ is necessarily a polynomial. The original form allows the version of (P1) in which $p$ is not necessarily a prime. Hall \cite{Hal71_OT} gave an example constructed by Woodall showing that the upper bound $e$ in (P2) is optimal. The reasoning behind this upper bound as well as the Hall-Woodall example is the (equivalent version of the) Prime Number Theorem stating that the product of primes up to $n$ is $e^{n+o(n)}$ and the fact that the residue class of $f(n)$ modulo this product is determined uniquely by $f(0),\ldots,f(n-1)$ thanks to (P1). In 1971, Hall \cite{Hal71_OP} and Ruzsa \cite{Ruz71_OC} independently proved the following result.
		\begin{theorem}[Hall-Ruzsa, 1971]
		Suppose that $f:\N_0\rightarrow \Z$ satisfies (P1) and
		$$\limsup_{n\to\infty}\frac{\log\vert f(n)\vert}{n}<e-1$$
		then $f$ is a polynomial.
		\end{theorem}
		
		The best upper bound was obtained in 1996 by 
		Zannier \cite{Zan96_OP} by 
		extending earlier work of Perelli and Zannier \cite{Zan82_AN,PZ84_OR}:
		\begin{theorem}[Zannier, 1996]
		Suppose that $f:\N_0\rightarrow \Z$ satisfies (P1) and
		$$\limsup_{n\to\infty}\frac{\log\vert f(n)\vert}{n}<e^{0.75}$$
		then $f$ is a polynomial.
		\end{theorem}
		In fact, the author remarked \cite[pp.~400--401]{Zan96_OP} that the explicit upper bound $e^{0.75}$ was chosen to avoid cumbersome formulas and it was possible to increase it slightly. The method of \cite{Zan96_OP} uses the fact that the generating series $\sum f(n)x^n$ is D-finite over $\Q$ (i.e. it satisfies a linear differential equation with coefficients in $\Q(x)$) \cite[Theorem~1.B]{PZ84_OR} then applies deep results on the arithmetic of linear differential equations \cite{CC85_AO,DGS94_AI}.
		
		This paper is motivated by our recent work on D-finite series \cite{BNZ19_DF} and a review of Ruzsa's conjecture. From now on, let $\bF$ be the finite field of order $q$
		and characteristic $p$, let $\cA=\bF[t]$, and let $\cK=\bF(t)$. 
		We have the usual degree map $\deg:\ \cA\rightarrow\N_0\cup\{-\infty\}$.
		A map
		$f:\ \cA\rightarrow \cA$ is called a polynomial map
		if it is given by values on $\cA$ of an element of $\cK[X]$. For every $n\in \N_0$, let $\cA_n=\{A\in\cA:\ \deg(A)=n\}$, $\cA_{<n}=\{A\in\cA:\ \deg(A)<n\}$,
		and $\cA_{\leq n}=\{A\in\cA:\ \deg(A)\leq n\}$. Let $\cP\subset\cA$ be the set
		of irreducible polynomials; the sets $\cP_{n}$, $\cP_{< n}$, and
		$\cP_{\leq n}$ are defined similarly. The superscript $+$ is used to denote the subset consisting of all the monic polynomials, for example $\cA^{+}$, $\cA_{n}^{+}$, $\cP_{\leq n}^{+}$, etc. From the well-known identity \cite[pp.~8]{Ros01_NT}:
		$$\prod_{d\mid n}\prod_{P\in\cP_{d}^+}P=t^{q^n}-t$$
        we have
        \begin{equation}\label{eq:q^n to 2q^n}
        q^n\leq \deg\left(\prod_{P\in\cP_{\leq n}^{+}}P\right)<2q^n
        \end{equation}
        for every $n\in\N$. In view of the reasoning behind Ruzsa's conjecture, it is natural to ask the following:
        
        \begin{question}\label{q:1}
        Let $f:\ \cA\rightarrow\cA$ satisfy the following 2 properties:
        \begin{itemize}
			\item [(P3)] $f(A+BP)\equiv f(A) \bmod P$ for every $A,B\in\cA$ and $P\in\cP$;
			\item [(P4)] $\displaystyle\limsup_{\deg(A)\to\infty}\frac{\log\deg(f(A))}{\deg(A)}<q$.         
        \end{itemize}
       Is it true that $f$ is a polynomial map?
        \end{question}
		
		Note that (P3) should be the appropriate analogue of (P1): over the natural numbers, iterating (P1) yields $f(n+bp)\equiv f(n) \bmod p$ for every $n,b\in\N_0$
		and prime $p$. On the other hand, over $\cA$, due to the presence of characteristic $p$, iterating the congruence condition $f(A+P)\equiv f(A) \bmod P$
		for $A\in\cA$ and $P\in\cP$ is not enough to yield (P3). By the following example that is similar to the one by Hall-Woodall, we have that the upper bound $q$ in (P4) cannot be increased. Fix a \emph{total} order $\prec$ on $\cA$ such that $A\prec B$ whenever $\deg(A)<\deg(B)$. We define $g:\ \cA\rightarrow\cA$ inductively. First, we assign arbitrary values of $g$ at the constant polynomials. Let $n\in\N$, $B\in \cA_n$, and assume that we have defined $g(A)$ for every $A\in\cA$ with $A\prec B$
		such that:
		$$g(A)\equiv g(A_1) \bmod P\ \text{for every $A,A_1\prec B$ and prime $P\mid(A-A_1)$}.$$
		For every $P\in \cP_{\leq n}^{+}$, let $R_P\in\cA$ with $\deg(R_P)<\deg(P)$
		such that $B\equiv R_P$ mod $P$. By the Chinese Remainder Theorem, there exists a unique $R\in \cA$ with $\deg(R)<\deg\left(\displaystyle\prod_{P\in\cP_{\leq n}^{+}}P\right)$ such that $R\equiv f(R_P)$ mod $P$ for every $P\in \cP_{\leq n}^{+}$. Then we define
		$$g(B):=R+\prod_{P\in\cP_{\leq n}^{+}}P.$$
		It is not hard to prove that $g$ satisfies Property (P3) (with $g$ in place of $f$) and for every $n\in\N$, $B\in\cA_n$, we have $\deg(g(B))\in [q^n,2q^n)$
		by \eqref{eq:q^n to 2q^n}. This latter property implies that
		$g$ cannot be a polynomial map.
		
		Our main result implies the affirmative answer to Question~\ref{q:1}; 
		in fact
		we can replace (P4) by the much weaker condition that 
		$\deg(f(A))$ is not too small compared to
		$\displaystyle\frac{q^{\deg(A)}}{\deg(A)}$:
		\begin{theorem}\label{thm:new main}
		Let $f:\ \cA\rightarrow\cA$ such that $f$ satisfies 
		Property (P3) 
		in Question~\ref{q:1}
		and 
		\begin{equation}\label{eq:1/27q}
		\deg(f(A))< \frac{q^{\deg(A)}}{27q\deg(A)}\ \text{when $\deg(A)$ is sufficiently large}.
		\end{equation} 
		Then $f$ is a polynomial map.
		\end{theorem}
		
	    There is nothing special about the constant $1/(27q)$ in
	    \eqref{eq:1/27q} and one can certainly improve it by optimizing the
	    estimates in the proof. It is much more interesting to know
	    if the function $q^{\deg(A)}/\deg(A)$ in \eqref{eq:1/27q} 
	    can be replaced by a larger function (see Section~\ref{sec:further questions}). There are significant differences between Ruzsa's conjecture and Question~\ref{q:1} despite the apparent similarities at first sight. Indeed none of the key techniques in the papers \cite{PZ84_OR,Zan96_OP} seem to be applicable in
	    our situation. Obviously, the crucial result used in \cite{Zan96_OP} that
	    the generating series $\sum f(n)x^n$ is $D$-finite has no counterpart here. 
	   The proof of the main result of \cite{PZ84_OR} relies on a nontrivial linear recurrence
	   relation of the form $c_df(n+d)+\ldots+c_0f(n)=0$. Over the integers, such a
	   relation will allow one to determine $f(n)$ for \emph{every} $n\geq d$  once 
	   one knows $f(0),\ldots,f(n-1)$. On the other hand, for Question~\ref{q:1}, while it seems possible to imitate the arguments in \cite{PZ84_OR} to obtain
	   a recurrence relation of the form $c_df(A+B_d)+\ldots+c_0f(A+B_0)=0$ for
	   $A\in\cA$ with $d\in\N$ and $B_0,\ldots,B_d\in\cA$, such a relation does not seem as helpful: when $\deg(A)$ is large, one cannot use the relation to relate $f(A)$ to
	   the values of $f$ at smaller degree polynomials. Finally, the technical trick of using the given congruence condition to obtain the vanishing on $[2M_0,(2+\epsilon)M_0]$ from the vanishing on $[0,M_0]$ (see \cite[pp.~11--12]{PZ84_OR} and \cite[pp.~396--397]{Zan96_OP}) does not seem applicable here.

	    The proof of Theorem~\ref{thm:new main} consists
		of 2 steps. The first step is to show that the 
		points $(A,f(A))$ for $A\in\cA$ belong to an
		algebraic plane curve over $\cK$, then it follows
		that $\deg(f(A))$ can be bounded above by 
		a linear function in $\deg(A)$. The second step,
		 which might be of independent interest,
		treats the more general problem in which $f$ 
		satisfies (P3) and there exists a special sequence
		$(A_n)_{n\in\N_0}$ in $\cA$ such that $\deg(f(A_n))$
		is bounded above by a linear function in $\deg(A_n)$. 
		Both steps rely on the construction of certain 
		auxiliary polynomials; such a construction has 
		played a fundamental role in 
		diophantine approximation, transcendental 
		number theory, and combinatorics. 
		For examples in number theory, 
		the readers are referred to \cite{BG06_HI,Mas16_AP} 
		and the references therein.	In combinatorics, the 
		method of constructing polynomials vanishing at 
		certain points has recently been called 
		the \emph{Polynomial Method} and is the subject 
		of the book \cite{Gut16_PM}. This method has 
		produced 
		surprisingly short and elegant solutions of 
		certain combinatorial problems over finite fields
		\cite{Dvi09_OT,CLP17_PF,EG17_OL}.
		
		\textbf{Acknowledgments.} We wish to thank Professor Umberto Zannier for useful discussions. J.~B is partially supported by an NSERC Discovery Grant. K.~N. is partially supported by an NSERC Discovery Grant, a start-up grant at UCalgary, and a CRC tier-2 research stipend.
	
	\section{A nontrivial algebraic relation}
	We start with the following simple lemma:
	\begin{lemma}\label{lem:for algebraic relation}
	Let $g:\ \cA\rightarrow\cA$ and assume that there exists $C_1\in\N_0$ such that the following 3 properties hold:
	\begin{itemize}
		\item [(a)] $g(A+BP)\equiv g(A) \bmod P$ for every $A,B\in\cA$ and $P\in\cP$.
		\item [(b)] $\deg(g(A))\leq q^{\deg(A)}-1$ for
		every $A\in\cA$ with $\deg(A)>C_1$.
		\item [(c)] $g(A)=0$ for every $A\in\cA_{\leq C_1}$. 
	\end{itemize}	 
	 Then $g$ is identically $0$.
	\end{lemma}
	\begin{proof}
	Otherwise, assume there is $A\in\cA$ of smallest degree
	such that $g(A)\neq 0$. We have $D:=\deg(A)>C_1$. Since $g(B)=0$ for
	every $B\in \cA_{<D}$ and since for every monic irreducible polynomial $P$ of degree at most $D$ there is some $C$ such that $A-CP$ has degree strictly less than $D$, we have
	$$g(A)\equiv 0 \bmod \prod_{P\in \cP_{\leq D}^+}P.$$
	Since $\deg\left(\displaystyle\prod_{P\in \cP_{\leq D}^+}P\right)\geq q^D$
	and $\deg(g(A))<q^D$, we must have $g(A)=0$, a contradiction.
	\end{proof}
		
	\begin{proposition}\label{prop:algebraic relation}
	Let $f:\ \cA\rightarrow\cA$ be as 
	in Theorem~\ref{thm:new main}. 
	Then there exists a non-zero polynomial $Q(X,Y)\in\cA[X,Y]$ such that 
	$Q(A,f(A))=0$ for every $A\in\cA$.
	\end{proposition}	
	\begin{proof}
	Let $N\in\N$ such that $\deg(f(A))<\displaystyle\frac{q^{\deg(A)}}{27q\deg(A)}$ for every $A\in\cA$ with $\deg(A)\geq N$. 
	Let $M\geq N$
	be a large positive integers that will be specified later.
	Consider $Q(X,Y)\in\cA[X,Y]$ of the form:
	$$Q(X,Y)=\sum_{0\leq i\leq q^M/3}\sum_{0\leq j\leq q^M/(3M)}\sum_{0\leq k\leq 9qM} c_{ijk}t^{i}X^jY^k$$
	where $c_{ijk}\in\F_q$. The number of unknowns 
	$c_{ijk}$ is greater than $q^{2M+1}$.
	
	Put $g(A)=Q(A,f(A))$ for $A\in\cA$ then $g$ satisfies
	the congruence condition:
	\begin{equation}\label{eq:congruence for g with the Q and cijk}
	g(A+BP)\equiv g(A) \bmod P\ \text{for every 
	$A,B\in\cA$ and $P\in\cP$}.
	\end{equation}
	
	We prove that with a sufficiently large choice of $M$, 
	we have
	$\deg(g(A))<q^M$
	for every $A\in \cA$ with $\deg(A)\leq M$. Suppose
	$\deg(A)\in [N,M]$ then we have:
	$$\deg(g(A))<\frac{q^M}{3}+\frac{q^M\deg(A)}{3M}+\frac{9qMq^{\deg(A)}}{27q\deg(A)}\leq q^M$$
	since the function $q^x/x$ is increasing on $[2,\infty)$.
	Now let $C_2$ be a positive number that is at least the maximum of $\deg(f(A))$ for $A\in \cA_{<N}$. Hence for every
	$A\in\cA_{<N}$, we have
	$$\deg(g(A))\leq \frac{q^M}{3}+\frac{Nq^M}{3M}+9C_2qM<q^M$$
	when $M$ is sufficiently large.

	Note that $\vert\cA_{\leq M}\vert=q^{M+1}$.
	Therefore the condition $g(A)=0$ for every $A$ with $\deg(A)\leq M$ is equivalent to the condition that the $c_{ijk}$'s satisfy
	a linear system of at most $q^{2M+1}$
	equations. Since the number of unknowns $c_{ijk}$ is greater than the number of equations, there exist $c_{ijk}$ not all zero 
	such that $g(A)=0$ for every $A\in \cA$
	with $\deg(A)<M$. 
	
	Finally, if $A\in\cA$ with $D:=\deg(A)>M$, we have
	$$\deg(g(A))\leq \frac{q^M}{3}+\frac{Dq^M}{3M}+\frac{Mq^D}{3D}<q^D$$
	since the function $q^x/x$ is increasing on $[M,\infty)$. Therefore the map
	$g:\ \cA\rightarrow\cA$ satisfies all the conditions of Lemma~\ref{lem:for algebraic relation} with $C_1=M$, we have that $g(A)=0$ for every $A\in\cA$ and this finishes the proof.
		\end{proof}
		
	\begin{corollary}\label{cor:linear bound}
	Let $f:\ \cA\rightarrow\cA$ be as in Theorem~\ref{thm:new main}. Then there exist $C_3,C_4>0$ depending only on $q$ and $f$ such that
	$$\deg(f(A))\leq C_3\deg(A)+C_4\ \text{for every $A\in\cA\setminus\{0\}$}.$$
	\end{corollary}
	\begin{proof}
	     By Proposition~\ref{prop:algebraic relation}, there exist
	     $n\geq 0$ and polynomials $P_0(X),\ldots,P_n(X)\in \cA[X]$
	     with $P_n\neq 0$ such that:
	     $$P_n(A)f(A)^n+P_{n-1}(A)f(A)^{n-1}+\ldots+P_0(A)=0$$
	     for every $A\in R$. We must have $n>0$ since otherwise $P_0(A)=0$ for every $A$ would force $P_0=0$ as well. Let $C_3=\max_{0\leq i\leq n}\deg(P_i)$ and let
	     $C_4$ be the maximum of the degrees of the coefficients of the $P_i$'s so
	     that $\deg(P_i(A))\leq C_3\deg(A)+C_4$ for every $A\in\cA\setminus\{0\}$.
	     If $\deg(f(A))>C_3\deg(A)+C_4$ then $\deg(P_n(A)f(A)^n)$ is greater
	     than $\deg(P_{n-1}(A)f(A)^{n-1}+\ldots+P_0(A))$, contradiction. 
	 \end{proof}	
		
	\section{A result under a linear bound}	 
		In this section, we consider a related result in which the inequality \eqref{eq:1/27q} is replaced by a much stronger linear bound on $\deg(f(A_n))$
		where $(A_n)_{n\geq 0}$ is a special sequence in $\cA$. Moreover, the next theorem together with Corollary~\ref{cor:linear bound} yield Theorem~\ref{thm:new main}.
		 
		\begin{theorem}\label{thm:linear bound}
		Let $f:\ \cA\rightarrow \cA$ satisfy the congruence condition 
		$$f(A+BP)\equiv f(A) \bmod P\ \text{for every $A,B\in\cA$ and $P\in\cP$}.$$ 
		Assume there exist $U\in \cA$ with $U'\neq 0$ (i.e. $U$ is not the $p$-th power of an element of $\bar{\F}[t]$) and positive integers
		$C_5$ and $C_6$ such that $\deg(f(U^n))\leq C_5n+C_6$
		for every $n\in\N_0$. Then $f$ is a polynomial map. 
		\end{theorem}

		For every non-constant $A\in\cA$, let $\rad(A)$ denote the product of the distinct monic irreducible factors of $A$.
		For integers $0\leq m<n$ and non-constant $U\in\cA$, let $\Delta_{m,n,U}=(U^n-1)(U^{n-1}-1)\ldots (U^{n-m}-1)$
		and let 
		$d_{m,n,U}=\deg(\rad(\Delta_{m,n,U}))$. We start with the following:
		
		\begin{lemma}\label{lem:product of t^i-1}
		Let $U(t)\in\cA$ such that $U'\neq 0$. Write $\delta=\deg(U)$. 
		\begin{itemize}
			\item [(a)] Let $M,\epsilon>0$. There exists a positive constant $C_7(\epsilon,M,p,U)$ depending only on $\epsilon$, $M$, $p$, and $U$ such that for every $n\geq 1$:
			$$\displaystyle d_{m,n,U}\geq \delta Mn^{2-\epsilon} - C_7(\epsilon,M,p).$$

			\item [(b)] Let $0\leq m<n$ be integers. There exist positive constants $C_8(p,U)$ depending only on $p$ and $U$ and	$C_9(m,p,U)$ depending only on $m$, $p$, and $U$ such that: 
			$$ d_{m,n,U}\geq \delta\left(1-\frac{1}{p}+\frac{1}{p^2}-\frac{1}{p^3}\right)mn-C_8(p,U)n-C_9(m,p,U).$$ 		
		\end{itemize}
		\end{lemma}
		\begin{proof}
		Since $U'\neq 0$, it has only finitely many roots. 
		For $\alpha\in\bar{\bF}$ that is not the value
		of $U$ at any of those roots, we have $\vert U^{-1}(\alpha)\vert=\delta$.
		
		For part (a), $d_{n-1,n,U}$ is at least the number of the preimages under $U$ of the roots of unity (in $\bar{\bF}^*$) whose order is at most $n$. For each $\ell$ with
		$p\nmid\ell$, there are exactly $\varphi(\ell)$ roots of unity of order $\ell$. Since $\varphi(\ell)$ dominates $\ell^{1-\epsilon}$, this proves part (a).
		
		For part (b), $d_{m,n,U}$ is at least the number of the preimages under $U$ of the roots of unity whose
		order divides $n-i$ for some $0\leq i\leq m$. Define:
		$$T=\{0\leq i\leq m:\ n-i\not\equiv 0\bmod p^2\}$$
		$$A_i=\{\zeta\in\bar{\bF}^*:\ \zeta^{n-i}=1\}\ \text{for each $i\in T$}.$$
		We have:
		$$d_{m,n,U}\geq \delta\vert \bigcup_{i\in T} A_i\vert+O_U(1)\geq \delta\left(\sum_{i\in T}\vert A_i\vert -\sum_{i,j\in T, i<j} \vert A_i\cap A_j\vert\right)+O_U(1).$$
		 
		Note that $\vert A_i\vert=\displaystyle\frac{n-i}{p^k}$ where $p^k\parallel n-i$. Let:
		$$S_0=\sum_{0\leq i\leq m}(n-i)=\frac{(2n-m)(m+1)}{2},$$
		$$S_1=\sum_{0\leq i\leq m, p\mid n-i} (n-i)=p\frac{(\lfloor n/p\rfloor+\lceil (n-m)/p\rceil)(\lfloor n/p\rfloor-\lceil (n-m)/p\rceil+1)}{2},$$
		\begin{align*}
		S_2&=\sum_{0\leq i\leq m, p^2\mid n-i} (n-i)\\
		&=p^2\frac{(\lfloor n/p^2\rfloor+\lceil (n-m)/p^2\rceil)(\lfloor n/p^2\rfloor-\lceil (n-m)/p^2\rceil+1)}{2}.
		\end{align*}
		
		We have:
		$$\sum_{i\in T}\vert A_i\vert=S_0-S_1+\frac{1}{p}(S_1-S_2)=\left(1-\frac{1}{p}+\frac{1}{p^2}-\frac{1}{p^3}\right)mn+O_{p}(1)n+O_{m,p}(1).$$
		For $i<j$ in $T$, we have $A_i\cap A_j\subseteq\{\zeta:\ \zeta^{j-i}=1\}$ hence $\vert A_i\cap A_j\vert\leq m$. Overall, we have
		$$d_{m,n,U}\geq \delta\left(1-\frac{1}{p}+\frac{1}{p^2}-\frac{1}{p^3}\right)mn+O_{p,U}(1)n+O_{m,p,U}(1)$$
		and this finishes the proof.
		\end{proof}
		
		We will need the following result on $S$-unit equations over characteristic $p$:
		\begin{proposition}\label{prop:S-unit characteristic p}
		Let $\Gamma\subset\cK^{*}$ be a finitely generated subgroup of rank $r$ 
		and consider the 
		equation $x+y=1$ with $(x,y)\in \Gamma\times\Gamma$. Then there 
		exists a finite subset $\scrX$
		of $\cK^*\times \cK^*$ of cardinality at most $p^{2r}-1$ 
		such that every solution $(x,y)\in (\Gamma\times\Gamma)\setminus (\bar{\bF}\times\bar{\bF})$ has the form
		$x=x_0^{p^k}$ and $y=y_0^{p^k}$
		for some $(x_0,y_0)\in \scrX$ and $k\in \N_0$.
		\end{proposition}
		\begin{proof}
		This is well-known; see \cite{Vol98_TN} or
		\cite[Proposition~2.6]{BN18_SF}.
		\end{proof}
		

		\begin{proof}[Proof of Theorem~\ref{thm:linear bound}]
		Recall that we are given $\deg(f(U^n))\leq C_5n+C_6$. 
		Let $\delta=\deg(U)$. 
		Let $N$, $D_1$, and $D_2$ be large positive integers 
		that will be specified later. Consider the auxiliary
		function:
		$$g(A)=P(A)f(A)+Q(A)$$
		where $Q(X)\in \cA[X]$ (respectively $P(X)\in \cA[X]$) has degree at most $D_1/\delta$ (respectively $(D_1-C_5)/\delta$)
		and each of its coefficients is an element of $\cA$ with degree at most $D_2$ (respectively $D_2-C_6$).  
		There are at least $q^{D_1D_2/\delta}q^{(D_1-C_5)(D_2-C_6)/\delta}$
		many choices for the pair $(P,Q)$. Note that 
		$g$ satisfies the congruence condition:
		$$g(A+BC)\equiv g(A) \bmod C\ \text{for every $A,B\in\cA$ and $C\in\cP$}.$$
		 
		We have
		$\deg(g(U^n))\leq D_1n+D_2$
		for every $n$. Hence there are at most
		$$\prod_{n=0}^N q^{D_1n+D_2+1}=q^{(D_1N(N+1)/2) + D_2(N+1)+N+1}$$
		possibilities for the tuple $(g(1),g(U),\ldots,g(U^N))$.	
		Fix a small positive $\epsilon$ that will be specified
		later. Now we choose a large $D_1$, then let:
		$$N+1=\frac{2-\epsilon}{\delta}D_1\ \text{and}\ D_2=\frac{\delta}{\epsilon}N(N+1),$$ 
		so that
		\begin{align*}
		\frac{D_1N(N+1)}{2}+D_2(N+1)+N+1
		&=\frac{1}{\delta}\left((\epsilon D_1D_2/2)+(2-\epsilon)D_1D_2+(2-\epsilon)D_1\right)\\
		&<\frac{1}{\delta}\left(D_1D_2+(D_1-C_5)(D_2-C_6)\right).
		\end{align*}
		
		By the pigeonhole principle, there exist two distinct choices
		of $(P,Q)$ giving rise to the same tuple
		$(g(1),\ldots,g(U^N))$. Taking the difference, we conclude that
		there exist such $P$ and $Q$ so that
		$g(U^i)=P(U^i)f(U^i)+Q(U^i)=0$ for $0\leq i\leq N$.
		For every $n>N$, we have $g(U^n)\equiv 0$ mod $\rad(\Delta_{N,n,U})$. Recall the constants $C_8(p,U)$ and $C_9(N,p,U)$ from 
		Lemma~\ref{lem:product of t^i-1}. Since
		$\displaystyle 1-\frac{1}{p}+\frac{1}{p^2}-\frac{1}{p^3}>\frac{1}{2}$, by choosing a sufficiently large $D_1$ (which implies that $N$ is sufficiently large)
		and sufficiently small $\epsilon$, we have:
		$$1-\frac{1}{p}+\frac{1}{p^2}-\frac{1}{p^3}-\frac{C_8(p,U)}{\delta N}>\frac{N+1}{(2-\epsilon)N}.$$
		This implies that
		for all sufficiently large
		$n$, we have:
		\begin{align*}
		\delta\left(1-\frac{1}{p}+\frac{1}{p^2}-\frac{1}{p^3}\right)Nn
		-C_8(p,U)n-C_9(N,p,U)&>\frac{\delta}{2-\epsilon}(N+1)n+D_2\\
		&=D_1n+D_2.
		\end{align*} 
		Since the right-hand side of the preceding inequality is at least $\deg(g(U^n))$ while
		the left-hand side is at most  $\deg(\Delta_{N,n,U})$ by 
		Lemma~\ref{lem:product of t^i-1}, we have $g(U^n)=0$ for
		all sufficiently large $n$. Let $N_1$ be such that
		$g(U^n)=0$ for every $n\geq N_1$.
		
		Now consider an arbitrary $A\in \cA\setminus\{0\}$ then fix an 
		integer
		$M>\deg(g(A))$. We claim that there
		exists $n\geq N_1$ such that $A-U^n$ has an irreducible factor $T$ of degree at least $M$; once this is done we have that $g(A)\equiv g(U^n)=0\, (\bmod\,T)$, and this forces $g(A)=0$, since the degree of $T$ is strictly larger than the degree of $g(A)$. To see why there exists such an irreducible factor $T$, let $\Gamma$
		denote the subgroup of $\cK^{*}$ generated by $U$, $A$, and all the irreducible polynomials of degree less than $M$. Since $U$ is not the $p$-th power of an element in
		$\bar{\F}[t]$, there exists an irreducible polynomial
		in $\cA$ whose exponent in the unique factorization of $U$ is not divisible by $p$, i.e. $v(U)\not\equiv 0$ mod $p$ where
		$v$ is the associated discrete valuation. 
		Therefore the set 
		$\scrS:=\{n\geq N_1:\ nv(U)-v(A)\not\equiv 0 \bmod p\}$
		is infinite and for every $n\in \scrS$, we have $U^n/A$ is not the $p$-th power of an element in $\cK$. Let $r$ denote the rank of $\Gamma$. Whenever $A-U^n=B$ has only irreducible factors of degree less than $M$, we have that
		$(U^n/A,B/A)$ is a solution of the equation $x+y=1$ with $(x,y)\in\Gamma\times\Gamma$. By Proposition~\ref{prop:S-unit characteristic p}, there can be at most $p^{2r}-1$ elements $n\in\scrS$ such that $A-U^n$ has only irreducible factors of degree less than $M$ and this proves our claim. 
		
		Hence $g(A)=0$ for every $A\in\cA\setminus\{0\}$ and the congruence condition on $g$ gives $g(A)=0$ for every $A\in\cA$. Hence $P(A)f(A)+Q(A)=0$ for
		every $A\in\cA$. We must have $P(X)\neq 0$; since otherwise $P(X)=Q(X)=0$. For all $A\in\cA$ except the finitely many $A$ such that $P(A)=0$, we have
		$Q(A)/P(A)=-f(A)\in\cA$. This implies that $P(X)\mid Q(X)$ in $\cK[X]$, hence
		$f$ is a polynomial map, as desired.
		\end{proof}

	\section{A further question}\label{sec:further questions}	
	As mentioned in the introduction, it is an interesting problem to 
	strengthen \ref{thm:new main} by replacing the function $q^{\deg(A)}/\deg(A)$
	in \eqref{eq:1/27q} by a larger function. Let
	$$d_n:=\deg\left(\prod_{P\in\cP_{\leq n}^+}P\right)$$
	which is the degree of the product of all monic irreducible polynomials
	of degree at most $n$. It seems reasonable to ask the following:
	\begin{question}
	Suppose $f:\ \cA\rightarrow\cA$ such that
	$f(A+BP)\equiv f(A)$ mod $P$ for every $A,B\in\cA$ and $P\in\cP$
	and there exists $\epsilon\in(0,1)$ such that for all sufficiently large $n$,
	for all $A\in\cA$ of degree $n$, we have
	$$\deg(f(A))\leq (1-\epsilon)d_n.$$
	Is it true that $f$ is a polynomial map?
	\end{question}
		
	\bibliographystyle{amsalpha}
	\bibliography{pRuzsa} 	

\end{document}